\documentclass[12pt,reqno]{amsart}

\usepackage{pb-diagram, graphics}
\usepackage{amscd,amsfonts}
\usepackage{amssymb}
\usepackage{amsmath}
\usepackage{amsthm}
\usepackage{dcpic,pictex}
\theoremstyle{plain}
\newtheorem{thm}{Theorem}[subsection]
\newtheorem{cor}[thm]{Corollary}
\newtheorem{lem}[thm]{Lemma}
\newtheorem{prop}[thm]{Proposition}

\theoremstyle{definition}
\newtheorem{defn}[thm]{Definition}

\newtheorem{ex}[thm]{Example}

\newcommand{\V}{\mathcal{V}}

\newcommand{\Irr}{\operatorname{Irr}}
\newcommand{\ev}{\operatorname{ev}}
\newcommand{\cx}{\operatorname{cx}}

\newcommand{\Ext}{\operatorname{Ext}}
\newcommand{\Hom}{\operatorname{Hom}}

\newcommand{\MaxSpec}{\operatorname{MaxSpec}}
\newcommand{\Ker}{\textrm{Ker}}

\newcommand{\ad}{\operatorname{ad}}

\newcommand{\gr}{\operatorname{gr}}

\newcommand{\0}{\bar 0}
\newcommand{\1}{\bar 1}

\newcommand{\HH}{\operatorname{H}}

\numberwithin{equation}{subsection}

\def\Z{{\mathbb Z}}

\def\:{\colon}

\newcommand{\fg}{\mathfrak{g}}

\def \fm{\mathfrak{m}}
\def \fg{\mathfrak{g}}

\def \fu{\mathfrak{u}}

\def\La{\mathfrak{g}}

\begin{document}

\title[Cohomology and support varieties for restricted  Lie superalgebras]
{Cohomology and support varieties for restricted  Lie superalgebras}\
\author{Irfan Bagci}
  \date{\today}

\address{Department of Mathematics \\
          University of California at  Riverside\\ Riverside, California 92521  }
\email{irfan@math.ucr.edu}


\begin{abstract} Let $\fg $ be a restricted Lie superalgebra over an algebraically closed field $k$ of characteristic $p>2$. Let $\fu(\fg)$ denote the restricted enveloping algebra of $\fg$.
In this paper we prove that the cohomology ring $\HH^\bullet(\fu(\fg), k)$ is finitely generated.  This allows one
to define support varieties for finite dimensional $\fu(\fg)$-supermodules. We also show that support varieties for finite dimensional $\fu(\fg)$
supermodules satisfy the desirable properties of support variety theory.
\end{abstract}

\maketitle

\parskip=2pt
\section{Introduction}
\subsection{} Throughout this paper we  assume that $k$ is an algebraically closed field of characteristic
 $p>2$.  All unspecified vector spaces, homomorphisms and tensor products are taken over $k$.  All vector spaces are assumed to be finite dimensional unless otherwise noted. Recall that a superspace is a $\Z_2$-graded vector space.  For a superspace $V= V_{\0} \oplus V_{\1}$, $\deg{v}$ will denote the $\Z_2$-degree of a homogeneous element $v \in V$.

 A superalgebra is a $\Z_2$-graded, unital, associative algebra $A = A_{\0}\oplus A_{\1}$ satisfying $A_iA_j \subseteq A_{i+j}$ for all $i, j \in \Z_2$. A Lie superalgebra is a finite dimensional
superspace $\La=\La_{\0}\oplus \La_{\bar{1}}$ with a bracket
 $[- ,- ] : \La\otimes \La \rightarrow  \La$ which preserves the $\Z_2$-grading
and satisfies graded versions of the operations used to define Lie algebras.
 Since the bracket preserves the $\Z_2$-grading, the even part  $\La_{\0}$ is a Lie algebra under the restriction of the  bracket and  the odd part $\fg_{\1}$ is a $\La_{\0}$-module under the bracket action. We view a Lie algebra as a Lie superalgebra concentrated in degree $\0$.  Given a Lie superalgebra ${\mathfrak g}$, $U(\La)$ will denote the universal enveloping algebra of $\fg$. $U(\fg)$ is a superalgebra and satisfies a PBW type theorem. See, for example, \cite{Kac, Sch}  for details and further background on Lie superalgebras.

A Lie superalgebra $\fg = \fg_{\0} \oplus \fg_{\1}$ is called restricted if the even part $\fg_{\0}$ is a restricted Lie algebra, i.e., it has a $p$th power map $()^{[p]} : \fg_{\0} \rightarrow \fg_{\0}$,  and the odd part $\fg_{\1}$ is a restricted $\fg_{\0}$-module under the adjoint action. If $\fg$ is a restricted Lie superalgebra, then the restricted enveloping algebra is defined to be the quotient algebra $\fu(\fg)= U(\fg)/J$, where $J$ is the two-sided ideal in the universal enveloping algebra $U(\fg)$ generated by the set $\{x^p-x^{[p]} : x \in \fg_{\0}\}$. A $\fg$-supermodule $M = M_{\0} \oplus M_{\1}$ is called restricted if $M$ is a supermodule for $\fu(\fg)$. For further details for the representation theory of restricted  Lie algebras and restricted Lie  super algebras we refer the reader to \cite{FPa1,FPa2, FPa3, FPa4,FPa5, WZ} . 
\subsection{}
 \subsection{} This paper is organized as follows. In Section 2    we review basic facts about  restricted  Lie superalgebras and super Hopf algebras  and record
the properties we are going to need in the rest of the paper.

In Section 3 we study the cohomology of $\fu(\fg)$.  We show that there exists a first quadrant spectral sequence converging to  the cohomology of the $\fu(\fg)$. From this spectral sequence,  it easily follows that $\HH^{\ev}(\fu(\fg), M)$ is a finitely generated $\HH^{\ev}(\fu(\fg), k)$-supermodule for every finite dimensional $\fu(\fg)$-supermodule $M$.

In Section 4, by using the finite generation results of  Section 2, we define support varieties.  An important property in the theory of support varieties for finite groups  is the realizability of any conical variety as  support variety of a supermodule. We prove an analogous  realization theorem.

\section{Basic facts and Results for restricted Lie superalgebras}
\subsection{} We begin with some basic definitions.
\begin{defn}A Lie superalgebra $\fg = \fg_{\0} \oplus \fg_{\1}$ is said to be restricted if the bracket $[, ]: \fg \times \fg \rightarrow \fg$ is supplemented with an additional operation $( \ )^{[p]} : \fg_{\0} \rightarrow \fg_{\0}$ called restriction satisfying:
  \begin{itemize}
  \item[(a)] $(cx)^{[p]}= c^px^{[p]}$ for all $c \in k$ and $x \in \fg_{\0}$,
  \item[(b)] For all $x \in \fg_{\0}$,  $\ad x^{[p]} = (\ad x)^p$,
  \item[(c)] $(x+y)^{[p]}= x^{[p]}+y^{[p]}+\Sigma_{i=1}^{p-1}s_i(x,y)$ for all $x, y \in \fg_{\0}$ where $is_i$ is the coefficient of $t^{i-1}$ in $(\ad (tx+y))^{p-1}(x)$.
  \end{itemize}
  \end{defn}
In short, a restricted Lie superalgebra is a Lie superalgebra whose
even subalgebra is a restricted Lie algebra and the odd part is a restricted module by the adjoint action of the even subalgebra.

  \begin{defn}Let $\fg$ be a restricted Lie superalgebra. Let $U(\fg)$ be the universal enveloping algebra of $\fg$. The restricted enveloping algebra $\fu(\fg)$ of $\fg$  is the quotient algebra
  $$\fu(\fg)= U(\fg)/J,$$
  where $J$ is the two sided ideal of $U(\fg)$ generated by $x^p-x^{[p]}$, $x\in \fg_{\0}$.
  \end{defn}

  Let $x_1, \dots, x_m $ and $y_1, \dots, y_n$ be  bases for $\fg_{\0}$ and $\fg_{\1}$, respectively. Then the  set
$$\{ x_1^{a_1} \dots x_m^{a_m}y_1^{b_1} \dots y_n^{b_n} \mid 0\leq  a_i<p;  \ b_j = 0,1  \ \ \text{for all} \ \  i, j \}$$
  is a basis for $\fu(\fg)$. In particular, $\dim \fu(\fg)= p^{\dim \fg_{\0}} 2^{\dim \fg_{\1}}.$

  The category of all $\fg$-supermodules is identified in a natural way with the category of all $U(\fg)$-supermodules. This construction identifies the full subcategory of
  all restricted $\fg$-supermodules with that of all $\fu(\fg)$-supermodules.

  \subsection{Super Hopf algebra structure of $U(\fg)$ and $\fu(\fg)$} To define a   a $\Z_2$-graded Hopf algebra (or super-Hopf algebra) we begin first of all with a $\Z_2$-graded associative  algebra (or superalgebra) $A$. We then  consider the braided tensor product algebra $A \overline{\otimes} A$, where $\overline{\otimes}$ is the usual tensor product but with elements of odd degree skew-commuting. This allows us to equip  $A$  with a coproduct:
  $$\overline{\Delta} : A \rightarrow A\overline{\otimes} A$$
  which is a superalgebra homomorphism from $A$ to the braided tensor product algebra  $A\overline{\otimes} A$:
  $$\overline{\Delta}(ab)= \Sigma (-1)^{\deg{a_2} \ \deg{b_1}}a_1b_1\otimes a_2b_2=\overline{\Delta}(a) \overline{\Delta}(b)$$
  for any $a, b $ in $A$, where $\overline{\Delta}(a)= \Sigma a_1 \otimes a_2, \overline{\Delta}(b)= \Sigma b_1 \otimes b_2$ and $a_2, b_1$ homogeneous.
  We emphasize here that this is exactly the central point of difference between the super and the ordinary Hopf algebraic structure. In an ordinary Hopf algebra $H$,
  the coproduct $\Delta : H \rightarrow H \otimes H$ is  an algebra homomorphism from $H$ to the usual tensor product algebra $H\otimes H$.

  Similarly, $A$ is equipped with an antipode $\overline{S} : A \rightarrow A$ which is not an algebra anti-homomorphism (as it should be in an ordinary Hopf algebra) but  a superalgebra anti-homomorphism (or twisted anti-homomorphism, or braided anti-homomorphism) in the following sense
  $$\overline{S}(ab) = (-1)^{\deg{a} \deg{b}} \overline{S}(a) \overline{S}(b),$$
  for any homogeneous $a, b \in A$.

  The rest of the axioms which complete the super-Hopf algebraic structure (i.e., coassociativity, counit property and compatibility with the antipode) have the same formal descriptions as in ordinary Hopf algebras.

  $U(\fg)$ and $\fu(\fg)$ are super Hopf algebras with the following operations:
 $$ \overline{\Delta} : \fu(\fg)  \rightarrow \fu(\fg)\otimes \fu(\fg) ; \ \   \overline{\Delta}(x)= 1\otimes x + x\otimes 1, \ \  x \in \fg,$$
  $$ \varepsilon : \fu(\fg) \rightarrow k ; \ \  \varepsilon(x) =0 , x \in \fg  \ \  \text{and} \ \  \varepsilon(1) =1,$$
  $$ \overline{S} : \fu(\fg) \rightarrow \fu(\fg) ; \ \ \  \overline{S}(x) = -x , \ \ \  x \in \fg.$$

  The Hopf superalgebras $\fu(\fg)$ and  $U(\fg)$ are  super-cocommutative.

  \begin{prop} Let $H$ be a finite dimensional super-cocommutative super Hopf algebra over $k$ and let $M$ be a finite dimensional $H$-supermodule. Then $M$ is a summand of $M\otimes M^\ast \otimes M$. Therefore the following are equivalent:

 \begin{itemize}
  \item[(a)] $M$ is projective.
\item[(b)]$M\otimes M^\ast$ is projective.
 \item[(c)]$M\otimes M$ is projective.
 \item[(d)]$M^\ast$ is projective.
\item[(e)]$M$ is injective.

  \end{itemize}

  \end{prop}
  \begin{proof}This is argued exactly as in \cite[Proposition 3.1.10]{Ben1}.
 \end{proof}

  \begin{cor}\label{C:injective}
If $H$ is a finite dimensional super cocummutative super Hopf algebra over $k$, then $H$ is self injective. That is, a finite dimensional $H$-supermodule is projective if and only if it is injective. In particular, $\fu(\fg)$ is self injective.
\end{cor}

 For another proof of self injectivity of $\fu(\fg)$ we refer the reader to  \cite{FPa3, WZ}.


  \section{Cohomology}

  \subsection{} Let $M$ be an $\fu(\fg)$-supermodule and
$$P_{\bullet} : \cdots \longrightarrow P_n\longrightarrow \cdots \longrightarrow P_0\longrightarrow M$$

 be a projective resolution of $M$. Then the cohomology of $\fu(\fg)$ with coefficients in $M$ is defined by
   $$\HH^i(\fu(\fg), M):= \Ext _{\fu(\fg)}^i(k, M) \ \ \ i \geq 0,$$
where $k$ is viewed as an $\fu(\fg)$-supermodule via the augmentation.

   We use the following notational convention
   $$\HH^{\bullet}(\fu(\fg), M): = \bigoplus_{i=0}^\infty \HH^{i}(\fu(\fg), M),$$
   $$\HH^{\ev}(\fu(\fg), M): = \bigoplus_{i=0}^\infty \HH^{2i}(\fu(\fg), M).$$
Since $\HH^{\bullet}(\fu(\fg), k)$ is graded commutative, $\HH^{\ev}(\fu(\fg), k)$ is a commutative $k$-algebra.

 A minimal projective resolution
$$P_{\bullet} : \cdots \longrightarrow P_n\longrightarrow \cdots \longrightarrow P_0\longrightarrow M$$ of $M$ is defined as follows:
We set $P_0 = P(M)$ the projective cover of $M$ and let $\Omega(M)$ be the kernel of the  surjective even supermodule homomorphism $P(M)\longrightarrow M$. Inductively set $P_n =P( \Omega^n(M))$ and $\Omega^n(M)=\Omega(\Omega^{n-1}(M))$.

  \begin{ex} Let $\fg = \mathfrak{gl}(1|1)$. $\fg$ consists of all $2\times 2$ matrices over $k$. The even subalgebra $\fg_{\0}$ is generated by
 \[
 E_{11} =
 \begin{bmatrix}
 1& 0 \\
  0 & 0
\end{bmatrix}, \ \ \ \
E_{22} =
 \begin{bmatrix}
 0& 0 \\
  0 & 1
\end{bmatrix}.
 \]
 The odd part $\fg_{\1}$ is spanned by
 \[
 E_{12} =
 \begin{bmatrix}
 0& 1 \\
  0 & 0
\end{bmatrix}, \ \ \ \
E_{21} =
 \begin{bmatrix}
 0& 0 \\
  1 & 0
\end{bmatrix}.
 \]

   The irreducible supermodules in the principal block are one dimensional and indexed by $L(\lambda | -\lambda)$ where $\lambda \in \Z$. The projective cover $P(\lambda | -\lambda)$ is four dimensional with three radical layers. The head and socle of $P(\lambda | -\lambda)$ are both isomorphic to $L(\lambda | -\lambda)$ and the second layer is isomorphic to $L(\lambda +1 | -\lambda -1)\oplus L(\lambda -1 | -\lambda +1) $. The minimal projective resolution of the trivial supermodule $L(0 | 0)$ is given by
  $$\cdots \rightarrow P(2 | -2) \oplus P(0 | 0) \oplus P(-2 | 2)\rightarrow P(1 | -1) \oplus P(-1 | 1)\rightarrow P(0 | 0)\rightarrow L(0 | 0)\rightarrow 0$$
  Applying $\Hom_{\fu(\fg)}(-,k)$ we get a long exact sequence
  $$0\rightarrow \Hom_{\fu(\fg)}(L(0 \mid 0),k)\rightarrow \Hom_{\fu(\fg)}( P(0 | 0),k)\rightarrow \Hom_{\fu(\fg)}(P(1 | -1) \oplus P(-1 \mid 1),k)\rightarrow \cdots $$
  Since the resolution is minimal and $k$ is irreducible, all differentials are zero. Thus the cohomology is simply the set of  cochains.
  \end{ex}







  \subsection{Finite Generation} Let $G$ be a finite group.  A classical result (see \cite{Ev}) states that the group cohomology ring $\HH^\bullet(G, k)$ is a finitely generated $k$-algebra. This fact is used to define support varieties of supermodules over the group algebra $kG$. For a finite-dimensional cocommutative Hopf algebra $A$ the graded cohomology ring $\HH^\bullet(A, k)$ is shown to be finitely generated in \cite{FS}. Similarly as above, this is used to define support varieties of supermodules over $A$. In this subsection  we prove that the cohomology ring $\HH^\bullet(\fu(\fg), k)$ of a  restricted enveloping algebra $\fu(\fg)$ is finitely generated.


\subsection{}Let $I= \Ker(\varepsilon : \fu(\fg) \rightarrow k)$ be the augmentation ideal of $\fu(\fg)$.  
  \begin{lem} \label{P:nilaug}The augmentation ideal $I$ of $\fu(\fg)$ is nilpotent.
  \end{lem}
   \begin{proof} First note that  $$\fu(\fg) \cong  \fu(\fg_{\0})\otimes \Lambda(\fg_{\1})=\fu(\fg_{\0}) \oplus_{r\geq 1} \fu(\fg_{\0})\otimes \Lambda^r(\fg_{\1})$$
   as $\Z_2$-graded vector spaces. Thus the augmentation ideal of $\fu(\fg)$ will be the direct sum of augmentation ideal of $\fu(\fg_{\0})$ and some nilpotent elements. Since the augmentation ideal of $\fu(\fg_{\0})$ is nilpotent(cf. \cite{Jan}), the result follows.
  \end{proof}
\subsection{}
  The powers of the augmentation ideal yield a finite filtration of $\fu(\fg)$:
  $$\fu(\fg) = I^0\supseteq I^1\supseteq \cdots$$
Let $\gr \fu(\fg)$ denote the associated graded superalgebra $\bigoplus_{r\geq 0} I^r/I^{r+1}.$ Then $\gr \fu(\fg)$ is a super commutative super Hopf algebra.

For any basis $\{y_1, \dots, y_s\}$ of $\fg_{\0}^\ast$ let
$$(\fg_{\0}^\ast)^p: = (y_1^p, \dots, y_s^p)$$
denote the ideal of the polynomial ring $S(\fg_{\0}^\ast)$ generated by $y_1^p, \dots, y_r^p$. Since our field has characteristic $p$ this ideal does not depend on the choice of the  basis.

\begin{lem}\label{L:gr} Let $\fg = \fg_{\0} \oplus \fg_{\1}$ be a restricted Lie superalgebra and $\fu(\fg)$ be the associated restricted enveloping algebra. Then
  $$\gr \fu(\fg)= \gr\fu(\fg_{\0})\otimes \Lambda (\fg_{\1}) \cong S(\fg_{\0}^{*})/(\fg_{\0}^{*})^p\otimes \Lambda(\fg_{\1}). $$
\end{lem}

\begin{proof} Recall that  $\fu(\fg) \cong  \fu(\fg_{\0})\otimes \Lambda(\fg_{\1})$ as a $\Z_2$-graded vector space. Since this tensor product is a braided tensor product of algebras we have,
\begin{align}\label{E:gr}
\gr \fu(\fg) \cong  & \gr(\fu(\fg_{\0})\otimes \Lambda(\fg_{\1}))  \\
\cong & \gr\fu(\fg_{\0})\otimes  \gr \Lambda(\fg_{\1}) & \\
 \cong &  \gr\fu(\fg_{\0})\otimes   \Lambda(\fg_{\1}).
\end{align}
Since $\fg_{\0}$ is a restricted Lie algebra it is also well known that (cf. \cite{FPa4})
\begin{equation} \label{E:resL}
\gr\fu(\fg_{\0}) \cong  S(\fg_{\0}^{*})/(\fg_{\0}^{*})^p.
\end{equation}
Combining  (\ref{E:gr}) and (\ref{E:resL}) we have,
$$\gr \fu(\fg) \cong S(\fg_{\0}^{*})/(\fg_{\0}^{*})^p\otimes \Lambda(\fg_{\1}). $$
\end{proof}

  \subsection{}We are now ready to describe the cohomology rings. We begin with:


\begin{prop}\label{P:Cgr} Let $\fg = \fg_{\0} \oplus \fg_{\1}$ be a restricted Lie superalgebra and $\fu(\fg)$ be the associated restricted enveloping algebra . Then
\begin{align*}
\HH^\bullet(\gr \fu(\fg), k)
  \cong & S(( \fg_{\0}\oplus \fg_{\1})^\ast)\otimes \Lambda(\fg_{\0}^\ast).
\end{align*}
\end{prop}
\begin{proof}It is well-known that the cohomology of a tensor product is
essentially the tensor product of the cohomologies. 
Thus  by Lemma \ref{L:gr} 
\begin{align}\label{E:cohotensor}
\HH^\bullet( \gr \fu(\fg), k) & \cong    \HH^\bullet( S(\fg_{\0}^{*})/(\fg_{\0}^{*})^p\otimes \Lambda(\fg_{\1}), k) \\ & \cong \HH^\bullet( S(\fg_{\0}^{*})/(\fg_{\0}^{*})^p,k)\otimes
\HH^\bullet(  \Lambda(\fg_{\1}), k)
\end{align}
By \cite[Theorem 5.1]{FPa2}
\begin{equation}\label{E:resl}
\HH^\bullet( S(\fg_{\0}^{*})/(\fg_{\0}^{*})^p,k) \cong S( \fg_{\0} ^\ast)\otimes \Lambda(\fg_{\0}^\ast)
\end{equation}
and by  \cite[Proposition  3.6]{AAH}
 \begin{equation}\label{E:sym} 
\HH^\bullet(\Lambda(\fg_{\1}), k) \cong    S(\fg_{\1}^\ast).
\end{equation}
Putting together (\ref{E:cohotensor}), (\ref{E:resl}) and (\ref{E:sym}), one has
$$\HH^\bullet(\gr \fu(\fg), k)
  \cong  S( \fg_{\0} ^\ast)\otimes \Lambda(\fg_{\0}^\ast)\otimes S(\fg_{\1}^\ast) \cong S(( \fg_{\0}\oplus \fg_{\1})^\ast)\otimes \Lambda(\fg_{\0}^\ast). \qedhere$$
\end{proof}

We now introduce some spectral sequence :
\begin{thm}\label{T:spec} Let $\fg = \fg_{\0}\oplus \fg_{\1} $ be a restricted Lie superalgebra and $M=M_{\0}\oplus M_{\1}$ be a finite dimensional restricted $\fg$-supermodule. Then there exists a spectral sequence
$$E_0^{i, j}(M) = S^l( \fg_{\0}^\ast) \otimes \Lambda^m(\fg_{\0}^\ast)\otimes S^n( \fg_{\1}^\ast)\otimes M \Longrightarrow H^{i+j}(\fu(\fg), M),$$
where i+j = 2l+m+n.
\end{thm}

\begin{proof} 
 One can compute the cohomology ring  $\HH^\bullet(\fu(\fg), M)$ by using cobar resolution:
The cochain are  defined by $$C^0(M) = M$$ and $$C^r(M)= I^{\otimes r}\otimes M$$ for $r>0$, where $I$ is the augmentation ideal of $\fu(\fg)$ and $I^{\otimes r}$ denotes the tensor product of $I$ with itself $r$ times. 
The differentials are defined by  $$d_0: C^0(M) \rightarrow C^1(M)$$ equal to the zero map; for $r>0$, define $$d_r: C^r(M) \rightarrow C^{r+1}(M)$$ by the following formula
$$ d_r(x_1\otimes  \dots \otimes  x_r\otimes m) = \sum_{1\leq i\leq r; \ 1\leq j \leq j(i)}(-1)^i (x_1\otimes  \dots\otimes  x_{i-1}\otimes x_{ij}\otimes  x_{ij}'\otimes  x_{i+1}\otimes \dots\otimes  x_r\otimes  m)$$
where $\overline{\Delta}(x_i) = \sum x_{ij}\otimes x_{ij}'$ denotes the comultiplication of $\fu(\fg)$.

 The finite  filtration $$\fu(\fg) = I^0\supseteq I^1\supseteq \cdots$$of $\fu(\fg)$ gives a filtration of the cochains. Then by the general theory of filtered complexes there is a spectral sequence with $E_1$-terms $$E_1^{i,j}(M) = \HH^{i+j}(\gr \fu(\fg), k)_i\otimes M.$$ Since the augmentation ideal $I$ of $\fu(\fg)$ is nilpotent by Proposition \ref{P:nilaug}, the spectral sequence converges to the cohomology of the original complex. That is,
\begin{equation}\label{E:SE1}
E_1^{i,j}(M) = \HH^{i+j}(\gr \fu(\fg),k)_i\otimes M \implies \HH^{i+j}(\fu(\fg), M).
\end{equation}

Combining (\ref{E:SE1}) with Proposition \ref{P:Cgr} we see that our spectral sequence can be written as follows
\begin{equation}\label{E:E1}
E_1^{i, j}(M) = \HH^{i+j}(\gr \fu(\fg), M)
\cong S^l( \fg_{\0}^\ast) \otimes \Lambda^m(\fg_{\0} ^\ast)\otimes S^n( \fg_{\1}^\ast)\otimes M,
\end{equation}
where  $i+j = 2l+m+n$. Thus  $$E_1^{i(p-1)+j, -(p-2)j}(M) = S^l( \fg_{\0}^\ast) \otimes \Lambda^m((\fg_{\0}^\ast)\otimes S^n( \fg_{\1}^\ast)\otimes M  ,$$ where $i+j = 2l+m+n$,  and all other $E_1^{i,j}$ are zero. In particular $E_1^{i, j} =0$ for $(p-2)\nmid j$. This implies that $$d_r^{i,j} : E_r^{i, j}(M)\longrightarrow E_r^{i+r, 1-r+j}(M)$$ is zero for $r \not \equiv 1 \mod (p-2)$.  We can now reindex the spectral sequence by calling the new $E_r^{i, j}$ to be the old $E_{(p-2)r+1}^{i+j, -(p-2)i}$. This gives $E_0^{i, j}(M)$ as above.
\end{proof}
We can finally prove our finite generation result.
\begin{thm} \label{T:finiteg}Let $\fg = \fg_{\0}\oplus \fg_{\1} $ be a restricted Lie superalgebra and $M=M_{\0}\oplus M_{\1}$ be a finite dimensional restricted $\fg$-supermodule. Then
\begin{enumerate}
\item[(a)] The cohomology ring $\HH^\bullet(\fu(\fg), k)$ is finitely generated as a $k$-algebra.
\item[(b)] $\HH^\bullet(\fu(\fg), M)$ is a finitely generated $\HH^\bullet(\fu(\fg), k)$-supermodule.
\end{enumerate}
\end{thm}
\begin{proof}

(a) By Theorem \ref{T:spec}, as $\Lambda(\fg_{\0}^\ast)$ is finite dimensional, we observe that $$E_0^{\bullet, \bullet}(k):= \bigoplus_{i, j  \geq 0}E_0^{i, j}(k)$$ is  finitely generated as an $S(( \fg_{\0}\oplus \fg_{\1})^\ast) \cong S( \fg_{\0}^\ast)\otimes  S( \fg_{\1}^\ast) $-supermodule. Since
    $E_\infty(k)$ is a section of $E_0^{\bullet, \bullet}(k)$,  it is also finitely generated as an  $S(( \fg_{\0}\oplus \fg_{\1})^\ast)$ supermodule. Since $S(( \fg_{\0}\oplus \fg_{\1})^\ast)$  is a finitely generated $k$-algebra,  by transitivity of  finite generation it follows that $E_\infty(k)$ is finitely generated as a $k$-algebra.

(b) By \cite[Theorem 4]{May}, $E_0^{\bullet, \bullet}(M)$ is a differential supermodule for the differential algebra $E_0^{\bullet, \bullet}(k)$. By arguing as in part (a), one sees that $$E_0^{\bullet, \bullet}(M):= \bigoplus_{i, j \geq 0}E_0^{i, j}(M)$$ is finitely generated as an $S(( \fg_{\0}\oplus \fg_{\1})^\ast)\subseteq E_\infty(k)$-supermodule. Thus $E_\infty(M)$ is finitely generated as an $E_\infty(k)$-supermodule. The finiteness of $\HH^\bullet(\fu(\fg), M)$ over $\HH^\bullet(\fu(\fg), k)$ is argued as in \cite[Proposition 2.1]{Ev}.
\end{proof}

\section{support varieties}
\subsection{} In this section we  recall the notion of the support variety of a finite dimensional $\fu(\fg)$-supermodule and study the properties of these varieties.  Let $M, N$ be finite dimensional  $\fu(\fg)$-supermodules. Recall that  $\HH^{\ev}(\fu(\fg), k)$ acts on $\HH^{\bullet}(\fu(\fg), M^\ast \otimes N)$. Let
$I(M, N)$ be the annihilator ideal of this action. We define the   \emph{relative support variety} of the pair
$(M,N)$ to be
$$\V_{\fg}(M, N) = \MaxSpec (\HH^{\ev}(\fu(\fg), k)/I(M, N)),$$
the maximal ideal spectrum of the finitely generated commutative $k$-algebra $$\HH^{\ev}(\fu(\fg), k)/I(M, N)$$.  For short hen $M=N$,  we write $$I(M):=I(M,M)$$ and
$$\V_\fg(M):= \V_\fg(M,M).$$
The latter is called the \emph{support variety} of $M$.

Since $\HH^{\bullet}(\fu(\fg), M^\ast \otimes M)$ is a graded supermodule over the graded ring $\HH^{\ev}(\fu(\fg), k)$, the variety $\V_\fg(M)$ is a closed, conical subvariety of $\V_\fg(k)$.

Note that for any finite dimensional $\fu(\fg)$ supermodules $M$ and $N$ and any maximal ideal $\fm$ in $\HH^{\ev}(\fu(\fg), k)$ we have
\begin{equation}\label{E:SI}
\fm \in \V_\fg(M, N) \ \  \text{if and only if} \ \  I(M, N)\subseteq \fm  \ \ \text{if and only if} \ \  \Ext^\bullet(M, N)_{\fm} \neq 0.
\end{equation}

\subsection{}
The following theorem shows that support varieties for finite dimensional $\fu(\fg)$ supermodules satisfy the desirable properties of support variety theory.

\begin{thm} \label{T:SP}Let $M, N$ be finite dimensional $\fu(\fg)$-supermodules. Then,
\begin{itemize}

\item [(a)] $\V_\fg(M\oplus N)= \V_\fg(M) \cup \V_\fg(N)$.
\item [(b)] $\V_{\fg}(M, N)\subseteq \V_\fg(M) \cap \V_\fg(N).$
\item [(c)] $\V_\fg(M)= \cup_{S \in \Irr(\fu(\fg))}\V_\fg(M, S)=\cup_{S \in \Irr(\fu(\fg))}\V_\fg(S, M)$, where $\Irr(\fu(\fg))$ denotes the set of all irreducible $\fu(\fg)$-supermodules.

    \item [(d)] $\V_\fg(M)= \{0\}$ if and only if $M$ is projective.
\item [(e)] $\V_\fg(M\otimes N)\subseteq \V_\fg(M) \cap \V_\fg(N)$.
\end{itemize}
\end{thm}
\begin{proof}

(a)-(c) is proven as in \cite[Section 5.7]{Ben2}.

(d) If $M$ is projective, then higher extension of $M$ will vanish. Therefore,  $\V_\fg(M)= \{0\}$.
If  $\V_\fg(M)= \{0\}$, then by (c) $\V_\fg(S, M)= \{0\}$ for every irreducible $\fu(\fg)$-supermodule $S$. Since $\Ext^\bullet(S,
M) \cong \HH^\bullet(\fu(\fg),S^\ast\otimes M )$ is finitely generated as $\HH^\bullet(\fu(\fg), k)$-supermodule by Theorem \ref{T:finiteg}, there exist some integer $K$ such that $\Ext^n(S,M)=0$ for all $n>K$ and all irreducible supermodules $S$. So the minimal injective resolution of $M$ is finite. Since $\fu(\fg)$ is self injective by Corollary \ref{C:injective} the result follows.

(e) The proof is fairly standard but we include the details for the reader's convenience. As in the case of finite groups  the action of $\HH^{\ev}(\fu(\fg), k)$ on $\Ext^\bullet_{\fu(\fg)}(M\otimes N, M\otimes N)$ is composed of applying first $-\otimes M$ then $-\otimes N$ and finally applying the Yoneda product. Therefore $I(M)\subseteq I(M\otimes N)$,  and this implies that $\V_\fg(M\otimes N)\subseteq \V_\fg(M)$. Since $\Ext^\bullet_{\fu(\fg)}(M\otimes N, M\otimes N)\cong \Ext^\bullet_{\fu(\fg)}(M, M\otimes N\otimes N^\ast)$, we have $\V_\fg(M\otimes N)=\V_\fg(M, M\otimes N\otimes N^\ast)$ and by (b) this is included in $ \V_\fg(M)$. Similarly $\V_\fg(M\otimes N)\subseteq \V_\fg(N)$.

\end{proof}

\subsection{}

One important property in the theory of support varieties is the realizability
of any conical variety as the support variety of some supermodule in the category. Carlson \cite{Ca1} first
proved this for finite groups in the 1980s. Friedlander and Parshall \cite{FPa2}
later used Carlson's proof to establish realizability for restricted
Lie algebras. For arbitrary finite group schemes the finite
generation of cohomology due to Friedlander and Suslin \cite{FS} allowed one to
define support varieties. In this generality the realizability of supports
was established using Friedlander and Pevtsova's method \cite{FPe} of concretely
describing support varieties through $\pi$-points.

\subsection{} We first describe our set up and prove some preliminary results.

Let $0 \neq \zeta \in \HH^{n}(\fu(\fg), k)$. Since $$\HH^{n}(\fu(\fg), k) \cong \Hom_{\fu(\fg)}(\Omega^n(k), k),$$  where $\Omega^n(k)$ denotes the $n$th syzygy of $k$,
$\zeta$ corresponds to a surjective map $$\hat{\zeta}:\Omega^n(k) \rightarrow k.$$
We  set
$$ L_{\zeta} =  \Ker (\hat{\zeta}:\Omega^n(k)\rightarrow k )\subseteq \Omega^n(k).$$ Thus $L_{\zeta}$ is defined by the following short exact sequence:
$$0\longrightarrow L_{\zeta}\longrightarrow \Omega^n(k) \stackrel{\hat{\zeta}}\longrightarrow k \longrightarrow 0$$
The supermodules $ L_{\zeta}$  are often called $\lq \lq$Carlson supermodules".

Since $\Omega^n(M)$ is isomorphic to $\Omega^n(k)\otimes M$ up to some projective summands, for any finite dimensional supermodules $M$ and $N$ we have
\begin{equation}\label{E:degreeshift}
\Ext_{\fu(\fg)}^r(\Omega^n(k)\otimes M, N) \cong \Ext_{\fu(\fg)}^{n+r}(M, N).
\end{equation}

If $J$ is an ideal of some commutative ring $A,$ then let $\mathcal{Z}(J)$ be the variety defined by $J.$  That is,
\[
\mathcal{Z}(J) = \left\{\mathfrak{m} \in \MaxSpec (A) \mid J \subseteq \mathfrak{m}  \right\}.
\]  In particular, for $a \in A$ let $\mathcal{Z}(a)$ denote the variety defined by the ideal $(a).$

\begin{prop} \label{T:Mzeta} Let $M$ be a finite dimensional $\fu(\fg)$-supermodule, and let $\zeta$ be a nonzero homogeneous element of positive degree in $\HH^{\ev}(\fu(\fg), k)$. Then
$$\V_\fg(M\otimes L_{\zeta} ) = \V_\fg(M) \cap \mathcal{Z}(\zeta). $$
\end{prop}
\begin{proof}The proof in \cite[Proposition 3]{PW} works in our setting.  We include the details for the  convenience of the reader.
We first show that $\V_\fg(M) \cap \mathcal{Z}(\zeta)\subseteq \V_\fg(M\otimes L_{\zeta} )$.  By Theorem \ref{T:SP}(c), we have
$$\V_\fg(M)= \cup_{S \in \Irr(\fu(\fg))}\V_\fg(M, S) \ \  \text{and} \ \ \V_\fg(M\otimes L_{\zeta})= \cup_{S \in \Irr(\fu(\fg))}\V_\fg(M\otimes L_{\zeta}, S). $$
Therefore it is enough to show that
$$\V_\fg(M, S)\cap \mathcal{Z}(\zeta) \subseteq V_\fg(M\otimes L_{\zeta}, S) $$
for any irreducible $\fu(\fg)$ supermodule $S$.  Let $\fm$ be a maximal ideal in $\V_\fg(M, S)\cap \mathcal{Z}(\zeta)$. Then $I(M,S)\subseteq \fm$ and $(\zeta) \subseteq \fm$. Thus $\fm$ contains the ideal generated by $I(M,S)$ and $\zeta$. We would like to show that $\fm \in \V_{\fg}(M\otimes L_{\zeta} ,S)$, that is $I(M\otimes L_{\zeta} ,S)\subseteq \fm$.  Suppose that $I(M\otimes L_{\zeta} ,S)\nsubseteq \fm$. Then (\ref{E:SI}) implies that $\Ext_{\fu(\fg)}^\bullet(M\otimes L_{\zeta} ,S)_{\fm} = 0$. By first tensoring the short exact sequence
$$0\longrightarrow L_{\zeta}\longrightarrow \Omega^n(k)\longrightarrow k \longrightarrow 0$$
with $M$ and then applying $\Ext_{\fu(\fg)}(-, S)$, we get a long exact sequence

$$\cdots \longrightarrow \Ext_{\fu(\fg)}^r(M, S) \longrightarrow  \Ext_{\fu(\fg)}^r(\Omega^n(k)\otimes M, S)\ \longrightarrow  \Ext_{\fu(\fg)}^r(M\otimes L_{\zeta}, S) $$ $$\longrightarrow \Ext_{\fu(\fg)}^{r+1}(M, S)\ \longrightarrow \Ext_{\fu(\fg)}^{r+1}(\Omega^n(k)\otimes M, S)\ \longrightarrow  \Ext_{\fu(\fg)}^{r+1}(M\otimes L_{\zeta} \longrightarrow \cdots$$

Using (\ref{E:degreeshift})  the long exact sequence above  can be written as follows:
$$\cdots \longrightarrow \Ext_{\fu(\fg)}^r(M, S) \stackrel{\zeta} \longrightarrow \Ext_{\fu(\fg)}^{n+r}(M, S)\stackrel{\psi} \longrightarrow  \Ext_{\fu(\fg)}^r(M\otimes L_{\zeta}, S)\longrightarrow \Ext_{\fu(\fg)}^{r+1}(M, S)\longrightarrow  \cdots$$

where the map  $\zeta  : \Ext_{\fu(\fg)}^r(M, S)\longrightarrow \Ext_{\fu(\fg)}^{n+r}(M, S)$ is just the action of
$\zeta \in \HH^{n}(\fu(\fg),k)$ on $ \Ext_{\fu(\fg)}^r(M, S)$.  All the maps in the sequence above are $\HH^{\ev}(\fu(\fg),k)$-supermodule homomorphisms. Let $z \in   \Ext_{\fu(\fg)}^{n+r}(M, S)$. Then $\psi(z) \in\Ext_{\fu(\fg)}^r(M\otimes L_{\zeta}, S)$. Since $ \Ext_{\fu(\fg)}^\bullet(M\otimes L_{\zeta}, S)_{\fm} = 0$, there exists a homogeneous element $a \notin \fm$ such that $\psi(az) = a\psi(z) =0$. Since the long cohomology sequence is an exact sequence $az = \zeta y$ for
 $y \in \Ext_{\fu(\fg)}^{r+\deg (a)}(M, S)$. Thus $z = \zeta a^{-1}y$. This implies that
 $$\Ext_{\fu(\fg)}^i(M, S)_{\fm} = \zeta \Ext_{\fu(\fg)}^i(M, S)_{\fm} $$
 for all $i>n$. Assume $z \in \Ext_{\fu(\fg)}^i(M, S)_{\fm}$ for $i \leq n$.  Let $b \notin \fm$.  One can multiply $z$ by a high enough power of $b$
 so that $\deg(b^mz) > n$. Then $b^mz \in \Ext_{\fu(\fg)}^\bullet(M\otimes L_{\zeta}, S)_{\fm}$ and hence $z \in \Ext_{\fu(\fg)}^\bullet(M\otimes L_{\zeta}, S)_{\fm}$
 as $b$ is invertible in $\Ext_{\fu(\fg)}^\bullet(M\otimes L_{\zeta}, S)_{\fm}$.

 Since $\zeta \in \fm$, and $\Ext_{\fu(\fg)}^\bullet(M, S)$ is finitely generated over $\HH^{\ev}(\fu(\fg), k)$ by Theorem \ref{T:finiteg}, Nakayama's Lemma
 implies that $\Ext_{\fu(\fg)}^\bullet(M, S)_{\fm}=0$. This contradicts the assumption  $I(M, S) \subset \fm$. We conclude that $I(M, S) \subset \fm$ and hence
 $\V_\fg(M, S)\cap \mathcal{Z}(\zeta) \subseteq V_\fg(M\otimes L_{\zeta}, S) $.

To prove the other containment $\V_\fg(M\otimes L_{\zeta} )\subseteq \V_\fg(M) \cap \mathcal{Z}(\zeta)$,  by Theorem \ref{T:SP} (e) it is enough to show that
$\V_{\fg}(L_{\zeta} )\subseteq  \mathcal{Z}(\zeta).$  By Theorem \ref{T:SP}(c) showing that $\V_{\fg}(L_{\zeta},S )\subseteq  \mathcal{Z}(\zeta)$
  for any irreducible $\fu(\fg)$-supermodule $S$ will suffice. Let $\mathfrak{m}$ be a maximal ideal of  $\HH^{\ev}(\fu(\fg), k)$ for which $\zeta \notin \mathfrak{m}$.
   Then the action of $\zeta$ induces an isomorphism on the localized ring $\Ext_{\fu(\fg)}^\bullet(k, S)_{\fm}$ as $\zeta$ is an invertible element of $\HH^{\ev}(\fu(\fg), k)_{\fm}$. Since localization is an exact functor the short exact sequence which defines the Carlson supermodule $L_{\zeta}$ implies that
   $\Ext_{\fu(\fg)}^\bullet(k, S)_{\fm}$ is the kernel of the isomorphism
$$\Ext_{\fu(\fg)}^{\bullet}(k, S)_{\fm} \longrightarrow  \Ext_{\fu(\fg)}^{\bullet+n}(k, S)$$

induced by the action of $\zeta$ on $\Ext_{\fu(\fg)}^\bullet(k, S)_{\fm}$ and thus $\Ext_{\fu(\fg)}^\bullet(k, S)_{\fm} =0 $. Now  by ( \ref{E:SI}), we have
 $\V_{\fg}(L_{\zeta},S )\subseteq  \mathcal{Z}(\zeta)$.
\end{proof}

\begin{lem}\label{C:zeta}
Let $\zeta, \zeta_1,\dots, \zeta_n \in \HH^{\ev}(\fu(\fg), k)$ be
homogeneous elements with corresponding Carlson supermodules
$L_{\zeta_1}, \dots, L_{\zeta_n}$. Then
\begin{itemize}
\item[(a)] $\V_\fg(L_{\zeta}) = \mathcal{Z}(\zeta )$
\item[(b)] $\V_\fg(L_{\zeta_{1}} \otimes \dots \otimes L_{\zeta_{n}})= \V_{\fg}(L_{\zeta_{1}}) \cap \dots \cap \V_{\fg}(L_{\zeta_{n}})$

\end{itemize}
\end{lem}

\begin{proof}
\begin{itemize}
\item[(a)] By Theorem \ref{T:Mzeta} $$\V_\fg(L_{\zeta})=\V_\fg(k\otimes L_{\zeta} ) = \V_\fg(k) \cap \mathcal{Z}(\zeta)=  \mathcal{Z}(\zeta )$$
\item[(b)] This follows from successively applying Theorem \ref{T:Mzeta} and part (a).
\end{itemize}
\end{proof}



\subsection{} We are now prepared to prove the realization theorem.
\begin{thm} Let  $X$ be a conical subvariety of $\V_{\fg}(k).$ Then there exists a finite dimensional $\fg$-supermodule $M$  such that
$$ \V_\fg(M)=X.$$
\end{thm}

\begin{proof}

Let $J=(\zeta _1, \dotsc , \zeta _n) \subseteq \HH^{\ev}(\fu(\fg), k)$ be the homogeneous ideal which defines the homogeneous variety $X$. That is,$$X=\mathcal{Z}(\zeta _1)\cap  \dotsb \cap \mathcal{Z}(\zeta _n).$$  Let $M=L_{\zeta _1}\otimes \dotsb \otimes L_{\zeta _n}$.
Applying Lemma \ref{C:zeta} one has
\begin{align*}
\V_{\fg}(M) &= \V_{\fg}(L_{\zeta _1}\otimes \dotsb \otimes L_{\zeta _n})\\
  &=\V_{\fg }(L_{\zeta_{1}}) \cap \dots \cap \V_{\fg}(L_{\zeta_{n}} )\\
  &= \mathcal{Z}(\zeta _1)\cap  \dotsb \cap \mathcal{Z}(\zeta _n)=X
  \end{align*}\qedhere
\end{proof}

\subsection{}Let $V= \bigoplus_{n=0}^\infty V_n$ be a graded vector space with finite-dimensional homogeneous components. The \emph{rate of growth} $r(V)$ of $V$ is defined to be the smallest positive integer $c$ such that $\dim V_n \leq Kn^{c-1}$ for some constant $K$ and all $n = 0, 1, \dots$ If no such $c$ exists, set $r(V)= \infty$.
For example, the rate of the growth of a polynomial $k[x_1, \dots, x_s]$ ring with $s$ variables is $s$. Therefore  a finitely generated $k$-algebra of Krull dimension $s$ has rate of growth equal to $s$.

The \emph{complexity} of an $\fu(\fg)$-supermodule $M$, denoted by $\cx_{\fu(\fg)}(M)$, is the rate of growth of a minimal projective resolution $P_{\bullet}$ of $M$.


We have:
\begin{prop} Let $M$ be a finite dimensional $\fu(\fg)$-supermodule. Then
$$\cx_{\fu(\fg)}(M) = \dim \V_{\fg}(M).$$
\end{prop}

\begin{proof} We compute 
$$\dim \V_{\fg}(M)= \dim (\HH^{\ev}(\fu(\fg), k)/I(M))=r(\HH^{\ev}(\fu(\fg), k)/I(M))= r(\Ext^{\bullet}_{\fu(\fg)}(M,M)).$$
Thus it is enough to show that $\cx_{\fu(\fg)}(M)=r(\Ext^{\bullet}_{\fu(\fg)}(M,M)).$
Let $$P_{\bullet} : \cdots \longrightarrow P_n\longrightarrow \cdots \longrightarrow P_0\longrightarrow M$$
be a minimal projective resolution of $M$. Then the multiplicity of the projective cover $P(S)$ of a irreducible $\fu(\fg)$-supermodule $S$  as  a direct summand of
$P_n$ is equal to
\begin{equation}\label{E:1}
\dim \Hom_{\fu(\fg)}(P_n, S)
\end{equation}

Since the resolution is minimal and $S$ is irreducible all maps are zero, i.e., every homomorphism $P_n \longrightarrow S$ is a cocycle and every coboundary zero. Thus,

\begin{equation}\label{E:2}
\Hom_{\fu(\fg)}(P_n, S) \cong \Ext_{\fu(\fg)}^n(M, S)
\end{equation}
Combining (\ref{E:1}) and (\ref{E:2}) we have,
\begin{align}\label{E:3}
\dim P_n &  =\sum_{S \in \Irr (\fu(\fg))}\dim P(S) . \dim \Ext_{\fu(\fg)}^n(M, S)
\end{align}

where $\Irr (\fu(\fg))$ denotes the set of all irreducible $\fu(\fg)$-supermodules.

From (\ref{E:3}) one easily observes that
\begin{equation}\label{E:4}
r(P_{\bullet})\leq \max \{r(\Ext_{\fu(\fg)}^\bullet(M, S)) \mid S \in \Irr (\fu(\fg)) \}.
\end{equation}
For any irreducible supermodule $S$ since $\Ext^{\bullet}_{\fu(\fg)}(M,S)$ is finitely generated as a supermodule over $\Ext^{\bullet}_{\fu(\fg)}(M,M)$, we have
\begin{equation}\label{E:5}
\max \{r(\Ext_{\fu(\fg)}^\bullet(M, S)) \mid S \in \Irr (\fu(\fg)) \} \leq r(\Ext_{\fu(\fg)}^\bullet(M, M)).
\end{equation}
Since $\Ext_{\fu(\fg)}^n(M, M)$ is a subquotient of $\Hom_{\fu(\fg)}(P_n, M)$, and this is a subspace of $\Hom_k(P_n, M)$, we have
\begin{equation*}
\dim \Ext_{\fu(\fg)}^n(M, M) \leq \dim P_n . \dim M
\end{equation*}
for every non-negative integer $n$. This implies that

\begin{equation}\label{E:6}
r(\Ext_{\fu(\fg)}^\bullet(M, M))\leq r(P_{\bullet})
\end{equation}

Putting together  (\ref{E:4}), (\ref{E:5}), and(\ref{E:6})  we have
$$r(P_{\bullet})\leq \max \{r(\Ext_{\fu(\fg)}^\bullet(M, S)) \mid S \in \Irr (\fu(\fg)) \} \leq r(\Ext_{\fu(\fg)}^\bullet(M, M)) \leq r(P_{\bullet}).$$

Therefore all the inequalities above are equalities.

Since $\V_\fg(M)= \cup_{S \in \Irr(\fu(\fg))}\V_\fg(M, S)=\cup_{S \in \Irr(\fu(\fg))}\V_\fg(S, M)$ by Theorem \ref{T:SP}(c),
$$\dim \V_\fg(M) = \max (\dim \V_\fg(M, S)).$$
From the definition of the support varieties it is clear that $\dim \V_\fg(M, S)$ is equal to the Krull dimension of
$\HH^{\ev}(\fu(\fg), k)/I(M,S)$.  Because $\HH^{\bullet}(\fu(\fg), M\otimes S)$ is a finitely generated and faithful module  as a $\HH^{\ev}(\fu(\fg), k)/I(M,S)$-module,the Krull dimension of $ \HH^{\ev}(\fu(\fg), k)/I(M,S)$ is equal to  the rate of growth of the
$$\bigoplus_{n\geq 0}\HH^{n}(\fu(\fg), k)/I(M, N).$$

 Now the  statement of the Theorem follows from the equality $\cx_{\fu(\fg)}(M)=r(P_{\bullet})$.
\end{proof}

\textit{Acknowledgements} The author would like to thank his Ph.D advisor  Daniel Nakano for suggesting the problem and advice throughout the years.  The author would like to thank Jonathan Kujawa,  Weiqiang Wang, and  Lei Zhao for helpful conversations. The author also would like to thank Gizem Karaali for comments and corrections on an earlier draft of this paper.

\end{document}